\documentclass{amsart}
\usepackage{graphicx}
\usepackage{amsfonts}
\usepackage{float}
\usepackage{algorithm}
\usepackage{algorithmic}

\newtheorem{defi}{Definition}

\newtheorem{conj}{Conjecture}

\newtheorem{lm}{Lemma}
\newtheorem{ex}{Example}

\newtheorem{prop}{Proposition}
\newtheorem{nota}{Notation}

\newtheorem{prob}{Problem}

\usepackage{xcolor}
\begin{document}

\title{Numerical approach for solving problems arising from polynomial analysis}

\author{Yousra Gati, Vladimir Petrov Kostov and Mohamed Chaouki Tarchi}

\address{Universit\'e de Carthage, EPT-LIM, Tunisie}
\email{yousra.gati@gmail.com}
\address{Universit\'e C\^ote d’Azur, CNRS, LJAD, France}
\email{vladimir.kostov@unice.fr}
\address{Universit\'e de Carthage, EPT-LIM, Tunisie}
\email{mohamedchaouki.tarchi@gmail.com}

\begin{abstract}

This paper deals with the use of numerical methods based on random root sampling techniques to solve some theoretical problems arising in the analysis of polynomials. These methods are proved to be practical and give solutions  where traditional methods might fall short. \\

{\bf Key words:} random polynomial; random roots; simulation of polynomial in one variable; sign
 pattern; Descartes' rule of signs\\

\end{abstract}
\maketitle
\section{Introduction}

In this work, we propose a numerical approach to solve problems arising from the analysis of polynomials. Specifically, we address three distinct problems. The first two are related to Descartes' rule of signs, while the third concerns the distance between the  critical points and midpoints of
zeros of hyperbolic polynomials. The common fundamental question of these problems is whether or not polynomials that satisfy certain criteria (such as the sign of their coefficients and the  Descartes' rule of signs in the first example) exist. In the following sections, we present each problem as well as  some known theoretical results and then propose a numerical algorithm to solve some unknown cases.  The numerical approach  is based on the use of random sampling  roots to construct random  polynomials satisfying the conditions imposed by the problem. The random polynomial  thus constructed is then tested to conclude if it is a good example. When this is the case, the realizability problem is resolved. If  not, we cannot conclude, but only deduce that the case  has a strong chance to be non-realizable.  It turns out that our approach is efficient in easily finding examples in the case of realizability and in identifying a priori non-realizable cases. Therefore, this is an  invaluable support  for theoretical studies. Especially since the method is easy to implement and computationally fast. All programs are developed in Python and run on an Intel(R) Core(TM) i5-6200U PC  CPU at 2.30 GHz, and 16 GB of RAM.


\section{Problem 1: Existence of polynomials respecting the Descartes' rule of signs and a given sign pattern.}
The  famous Descartes'  rule of signs (see\cite{Des}) states that the number of positive roots of a univariate polynomial with real coefficients does not exceed the number of sign changes in its sequence of coefficients. In this context, we focus only on polynomials whose coefficients are all non-zero.  
      
\begin{defi}
{\rm We define a {\it sign pattern} as an arbitrary ordered sequence of signs $\sigma_0=(+, \pm, \ldots, \pm)$ beginning with a $+$.
For each sign pattern $\sigma_0$, we denote its {\it Descartes' pair} $(c, p)$ as the pair of positive integers that respectively count the sign changes and the sign preservations in $\sigma_0$.}
\end{defi}

  The Descartes' pair provides an upper bound for the number of positive and negative roots of any polynomial of degree $d$ the signs of whose coefficients define the sign pattern $\sigma_0$. It is important to note that for any sign pattern $\sigma_0$, the relation $p + c = d$ always holds true. To a polynomial $Q(x)=x^d+\sum_{j=0}^{d-1}a_jx^j$ of degree $d$ corresponding to the sign pattern $\sigma=(+, {\rm sgn}(a_{d-1}), \dots, {\rm sgn}(a_{0}))$, we associate the pair $(pos, neg)$, which represents the number of its positive and negative roots counted with  multiplicity. In 1890, Fourier further elaborated  this rule in \cite{Fo}  by stating that the discrepancy between the count of positive roots $pos$ and the number of sign changes $c$ in the coefficients is a multiple of 2, which means that,
\begin{equation}\label{eqdefi1} 
\begin{array}{llll}
 pos\leq c~,& c-pos\in 2\mathbb{Z},&neg\leq p~,& p-neg\in 2\mathbb{Z}.\\ 
\end{array} 
\end{equation}

\begin{defi}
{\rm For a given sign pattern $\sigma$ with Descartes' pair $(c, p)$ we call $(pos, neg)$ a {\it compatible pair} for $\sigma$ if conditions (\ref{eqdefi1}) are satisfied.}
\end{defi}
 One might ask whether, given a sign pattern $\sigma$  and a compatible pair $(pos, neg)$, it is possible to find a real monic polynomial of degree $d$, the signs of whose coefficients  define the sign pattern $\sigma$ and which has exactly $pos$  positive  and exactly $neg$  negative roots. In this case, we say that the couple $(\sigma, (pos, neg))$ is {\it  realizable}. It turns out that for $d = 1$, $2$, and $3$, the answer is positive, but for $d = 4$, the answer is negative; this result is due to Grabiner, see \cite{Gr}. He showed that for degree $4$, the following  are the only couples $(\sigma, (pos, neg))$ that are {\it not realizable}, meaning that no polynomial of degree $4$ can have
sign pattern and pair $(pos, neg)$ equal to 

\begin{eqnarray}
\label{grabiner}
 ((+~-~-~-+), (0, 2)) \quad \text{or} \quad ((+~+~-~+~+), (2, 0)).
\end{eqnarray}

It is clear that the second case can be obtained from the first one by  the change of  the variable $x$ to $-x$. In order  to consider simultaneously such equivalent cases, all researchers working on this issue have adopted the following group action:

\begin{defi}
{ \rm One defines the natural $\mathbb{Z}_2 \times \mathbb{Z}_2$-action on the space of monic polynomials
(and as a consequence on the space of couples $(\sigma, (pos,neg))$ as well) as follows:}
\begin{enumerate}
{\rm 
\item The first generator $g_1$ acts by changing the signs of all monomials in second,
fourth etc. position which for polynomials means $Q(x)\mapsto (-1)^{d}Q(-x)$ ; the admissible pair $(pos, neg)$ becomes $(neg, pos)$. We multiply by $(-1)^d$ to obtain a monic polynomial.
\item  The second generator $g_2$ acts by reading the sign pattern from the right which for polynomials
means $Q(x)\mapsto Q^R(x)/Q(0)$, where $Q^R(x):=x^dQ(1/x)$. One divides by $Q(0)$ in order to obtain again a
monic polynomial; the pair $(pos, neg)$ remains $(pos, neg)$. The generators are two commuting involutions.}
\end{enumerate}
\end{defi}
\begin{nota}
  {\rm We denote by $\Sigma _{m_1,m_2,\ldots ,m_s}$, $m_k\in \mathbb{N}$,
    $m_1+\cdots +m_s=d+1$, the sign pattern beginning with a sequence of $m_1$
    signs $+$ followed by a sequence of $m_2$ signs $-$
    followed by a sequence
    of $m_3$ signs $+$ etc. Example:}
    $$(+,+,-,+,+,+,-,+,+,+)=\Sigma _{2,1,3,1,3}~.$$
  \end{nota}
The research of Grabiner in degree 4 has stimulated the interest of several mathematicians, who have published various articles addressing the following question: 
\begin{prob}
For a given degree $d$, what are the couples $(\sigma, (pos,neg))$ that are non-realizable? 
\end{prob}
The exhaustive answer to this question was provided by Albouy and Fu for $d=5$ and $6$ in \cite{AlFu}, as well as  J.~Forsg\aa rd and  al. for $d=7$ in \cite{FoKoSh}, and by J.~Forsg\aa rd and al. and Kostov for $d=8$, see \cite{FoKoSh,KoCzMJ}. Other non-realizable couples where the sign pattern with $c=2$ have also been addressed by several  mathematicians  in recent years, as shown in \cite{CGK1}. 
 For $c=3$, several articles focus on these cases, particularly those concerning the question of non-realizability for compatible pairs with {\rm min}($pos$, $ neg$)=1, as indicated in \cite{KoMB}, \cite{CGK},
 and \cite{GKTcrm}. The common point among these papers is that the most powerful analytic means to prove realizability as the degree increases is the {\textit{ concatenation lemma}} published in \cite{FoKoSh}. However, this lemma proves insufficient for several cases starting from degree $d=5$. This is why we have developed a numerical method that allows  to find polynomials realizing the pairs in question, even when concatenation does not prove conclusive.
\begin{lm}{(Concatenation Lemma)}\label{lmconcat}
Suppose that the
monic polynomials $P_1$ and $P_2$ of degrees $d_1$ and $d_2$, with
sign patterns represented in the form 
$(+,\sigma _1)$ and $(+,\sigma _2)$ respectively, realize
the pairs $(pos_1, neg_1)$ and $(pos_2, neg_2)$. Here $\sigma _j$
denotes what remains of the sign patterns when the initial sign $+$ is deleted.
Then

(1) if the last position of $\sigma _1$ is $+$, then for any $\varepsilon >0$
small enough, the polynomial $\varepsilon ^{d_2}P_1(x)P_2(x/\varepsilon )$
realizes the sign pattern $(+,\sigma _1,\sigma _2)$ and the compatible pair
$(pos_1+pos_2, neg_1+neg_2)$;

(2) if the last position of $\sigma _1$ is $-$, then for any $\varepsilon >0$
small enough, the polynomial $\varepsilon ^{d_2}P_1(x)P_2(x/\varepsilon )$
realizes the sign pattern $(+,\sigma _1,-\sigma _2)$ and the pair
$(pos_1+pos_2, neg_1+neg_2)$. Here $-\sigma _2$ is obtained from $\sigma _2$
by changing each $+$ by $-$ and vice versa.
\end{lm}
\subsection{Numerical Method and Algorithm}\label{numericalmethod}

Our numerical method is based on generating independent and identically uniformly distributed random roots. For a given degree $d$  and a specified couple (sign pattern, $(pos, neg)$), the program generates $pos + neg$ real numbers to create roots that satisfy the pair ($pos, neg$), and $d - pos - neg$  real numbers to form  $(d - pos-neg)$   conjugate complex pairs of roots. The code then calculates the coefficients of the polynomial and checks if it matches the sign pattern. If it does, the program stops and returns the result, i.e., the polynomial and its decomposition. If it does not match, the program continues and repeats the simulation until it finds a valid polynomial or until a maximum number of simulations  $N $ is reached. The code utilizes two arbitrary parameters: a real number  $\ell$ that defines the interval within which we will generate uniformly distributed random numbers and the integer $N$, which specifies the maximum number of simulations, typically set to a very large value.

\begin{algorithm}[H]
\caption{Realizable couples (sign pattern, (pos, neg)) for a given degree \( d \)}
\label{algo}
\begin{algorithmic}[1]
    \STATE Input  a degree \( d \), a sign pattern and a pair \( (pos, neg) \).
    \STATE Generate random uniform \( pos \) positive  and \( neg \) negative numbers (roots) on an interval $[-\ell,\ell]$.
    \STATE Generate random uniform \( d - pos - neg \) numbers to construct pairs of conjugate complex roots.
    \STATE Calculate the coefficients of  the polynomial of degree \( d \) constructed using roots computed in the previous steps.
    \STATE Test whether the coefficients correspond to the  sign pattern and stop if it is the case.
    \STATE Repeat this process \( N \) times, where \( N \) is a large number.
\end{algorithmic}
\end{algorithm}

\subsection{Numerical tests}

The program has been tested first for some known cases. We verified that it actually gives examples of polynomials when the realizability is already established. This constitutes a first validation of the method. We list here polynomials $Q_1$, $ Q_2$, $Q_3$, $Q_4$ and $Q_5$ obtained by our method respectively for  the cases $(\Sigma_{1,3,2},(0,3))$, $(\Sigma_{1,5,2},(0,3))$, $(\Sigma_{2,1,1,1,2,1},(3,0))$, $(\Sigma_{2,1,4,1},(3,0))$, and $(\Sigma_{2,1,3,2},(3,0))$, whose realizability is established in \cite{AlFu, KoSha} and for which  no methodology or technique has been clearly presented. The results are rounded for clarity.

\begin{eqnarray*}
Q_1&:=&(x+0.723)(x+0.59)(x+0.48)(x^2-1.97x+0.977)\\
& =& x^5-0.177x^4-1.498x^3-0.125x^2+0.629x+0.2,\\
& & \\
Q_2&:=&(x+0.8)(x+0.77)(x+0.39)(x^2-0.13x+0.65)(x^2-1.88x+0.89) \\
&=& x^7-0.05x^6-0.927x^5-0.069 x^4-0.334x^3-0.08x^2 +0.389x+0.139,\\
& & \\
Q_3&:=&(x+0.389)(x-0.4121)(x-0.579)(x^2+1.4124x +0.499)(x^2+0.032x+0.704),\\
&=& x^7 + 0.065x^6 - 0.121x^5 + 0.096x^4 - 0.398x^3+0.030x^2 + 0.125x - 0.033,\\
& & \\
Q_4&:=&(x-0.5)(x-0.596)(x-0.975)(x^2+1.954x+0.956)(x^2+0.2x+0.359)\\
&=& x^7+0.083x^6-1.389x^5+0.013x^4+0.2x^3+0.014x^2 +0.210x-0.1,\\
& & \\
Q_5&:=&(x-0.597)(x-0.69)(x-0.85)(x^2+1.81x+0.83)(x^2+0.35x+0.15)\\
&=& x^7+0.023x^6-1.497x^5+0.017x^4+0.597x^3+0.0153x^2 -0.009x-0.044.\\
\end{eqnarray*}

The considered interval $[-\ell,\ell]$ for these tests is equal to $[-1,1]$.  The CPU time  is about $3.46$ seconds with $N=10^7$. 
 For higher degree tests, if we don't obtain results, one should change these parameter values  before concluding that the case is potentially non-realizable. We can for example increase the value of $N$ and/or choose a different  interval for the sampling of roots. That is,  some  roots will be chosen from a wide
uniform distribution and the remainder will be chosen from a much narrower distribution. 
Moreover,  one can consider a multiple roots polynomial sampling to improve   the approach. \\

The robustness of the method already proven, the algorithm is tested for some unknown cases  $$ C_1 := (\Sigma_{1,3,2,3,1}, (0, 3)),\; C_2 := (\Sigma_{1,3,1,3,2}, (0, 3)) \;\mbox{and}\; C_3:=(\Sigma_{1,5,1,1,2}, (0, 3)). $$ Our numerical method enables us to quickly generate concrete polynomials that effectively realize these couples and to conclude the realizability of these cases. Here are some examples:
\vspace{5mm}
\begin{eqnarray*}
P_{C_1}&:=&(x+0.2)(x+0.3)(x+0.3)\\
& & \qquad\qquad \qquad (x^2-1.84x+0.846)(x^2-1.62x+0.67)(x^2+1.72x+1.348)\\
&=&x^9-0.73x^8-1.5258x^7-0.191020x^6+2.52140544x^5-0.2081491476x^4\\
& &-1.051144849x^3-0.0043084783x^2+0.1632999587x+0.02862830065,\\
& & \\
P_{C_2}&:=&(x+0.25)(x+0.27)(x+0.43)\\
& & \qquad\qquad \qquad (x^2-0.4x+0.0402)(x^2-1.06x+0.33)(x^2+0.4x+0.14) \\
&=&x^9-0.11x^8-0.3659x^7-0.0082x^6+0.06757x^5+0.025x^4-0.0022x^3\\
& &-0.0022x^2-0.000017x+0.000053.\\
& & \\
P_{C_3}&:=&(x+ 0.786)(x+ 0.696)(x+ 0.622)\\
& & \qquad\qquad \qquad (x^2 + 0.3848x+ 0.808)(x^2 -0.5783x+ 0.706)(x^2 -1.972x + 0.975) \\
&=&x^9-0.0615x^8-0.43929984x^7-0.200085009x^6-0.798790981x^5-0.0587716x^4 \\
& & +0.044796444x^3 - 0.008446381x^2 + 0.369292280x + 0.1892530328.
\end{eqnarray*}

The CPU time with $N=10^7$  and $\ell=1$ is about $2.45$ seconds for $C_1$ and $C_2$ and about 200 seconds for $C_3$. It is important to note that the realizability of $C_3$ cannot be established using the concatenation lemma. Indeed,  the only options for applying the concatenation lemma are :
$$\begin{tabular}{llll}
$C_3$ &$(\Sigma_{1,5,1,1,1},(0,2))$ & {\rm and} & $(\Sigma_{2},(0,1))$ \\
 & $(\Sigma_{1,5,1},(0,2))$ & {\rm and} & $(\Sigma_{1,1,2},(0,1))$ 
\end{tabular}$$
However, this proves impossible, as in each configuration,  the first considered couple is non-realizable (see \cite[Theorem 10]{FoKoSh} and \cite{AlFu}). 
\section{Problem 2: Descartes' rule of signs and moduli of roots}

A real degree $d$ polynomial $ Q := \sum_{j=0}^{d} a_j x^j$, with  $a_d = 1 $, is said to be {\it hyperbolic } if all its roots are real. We assume that all coefficients $a_j$  are non-zero. In this context, Descartes' rule of signs tells us that this polynomial has $ c $ positive roots and $ p$  negative roots (counted with  multiplicity), with  $ c + p = d$. Here, $c$  represents the number of sign changes and $p$ the number of sign preservations in the sequence of coefficients of $Q$.

To explore the issue of realizability, we are interested in couples consisting of a sign pattern and an   {\it order of moduli}. The order of  moduli  is defined by the relative positions of the moduli of the positive and negative roots on the positive half-axis. A couple (sign pattern, order of moduli) is  {\it compatible }  if the order of moduli has exactly $ c$  moduli of positive  and $ p$ moduli of negative roots, all distinct.

We qualify such a couple as {\it realizable} if there exists a hyperbolic polynomial  the signs of whose coefficients and the moduli of whose roots define the sign pattern and the order of moduli of the couple. Thus, the question that arises can be formulated as follows:
\begin{prob}\label{prob2}
For a given degree $ d$ , what are the realizable compatible couples (sign pattern, order of  moduli)?
\end{prob}
The answer to this question is developed in the work of Kostov in \cite{KoCMA23} for degrees $d \leq 5$, as well as for $d = 6$ with two sign changes in \cite{KoSoz19}. A comprehensive answer to this problem for $d = 6$ is provided by Gati et al. in \cite{GKTMC}. For more information on the subject, one can refer to the works of Kostov in \cite{KoPuMaDe} and those of Gati et al. in \cite{GKTGJM}. It is noteworthy that, in these papers, the most powerful method for demonstrating the realizability of such a couple for each degree increase is the concatenation of couples presented in \cite[subsection 2.5]{GKTMC}. This involves studying a couple (sign pattern,  order of moduli) by concatenating two sequences of signs as explained below. However, this concatenation method as well as all known analytic methods  proves insufficient. In the same way of Problem 1, we develop a second algorithm, which is a slight modification of the previous one,  and conclude for the realizability of some unknown cases.\\

To better understand Problem \ref{prob2} and grasp the importance of the numerical method in its resolution, especially when analytic methods are not effective, it is very useful to introduce the following definition, along with the corresponding notations and examples that follow. For more comprehensive introduction of the problem, one can see \cite{Ko}, \cite{KoSe}, \cite{KorigMO} and \cite{theseT}.
\begin{defi}
{\rm The {\em order of moduli} is defined by the roots of a given hyperbolic
  polynomial $Q$  as follows (the general definition should be
  clear from this example.) Suppose that $d=7$ and that there are
  four negative roots $-\gamma_4<-\gamma _3<-\gamma _2<-\gamma _1$
  and three positive roots $\alpha _1<\alpha _2<\alpha _3$
  (so $c=2$ and $p=5$), where
  $$\alpha _1<\gamma _1<\gamma _2<\gamma _3<\alpha _2<\gamma _4<\gamma_5,$$
  then we say that the roots define the order of moduli $PNNNPNN$, i.~e. 
  the letters $P$ and $N$ denote the relative positions of the moduli of
  positive and negative roots.}
  \end{defi}
\begin{nota}
{\rm
 Consider the case where $d = 7$ and there are $4$ negative roots, denoted as $  -\gamma_4 < -\gamma_3 < -\gamma_2 < -\gamma_1$, along with three positive roots, denoted as $\alpha_1 < \alpha_2<\alpha_3 $ (which implies that $c = 3 $ and $p = 4$). We assume that the moduli of the roots satisfy the following inequalities:
  $$ \alpha_1 < \gamma_1 < \alpha_2<\gamma_2 < \gamma_3 < \alpha_3 < \gamma_4 .  $$
\noindent
We denote by $u_1$, $u_2$, $u_3$ and $u_4$ the number of moduli of the negative roots located in the respective intervals $ (0, \alpha_1)$, $(\alpha_1, \alpha_2)$, $(\alpha_2, \alpha_3)$ and $(\alpha_3, +\infty)$. In this case, we indicate that the roots define the order of  moduli as $[0,1, 2, 1] $, which means that $u_1 = 0 $, $ u_2 = 1 $, $u_3=2$ and $ u_4 = 1 $.
  }
\end{nota}
In the following paragraph, we will present the concatenation method, inspired by the concatenation lemma and first introduced in \cite{GKTMC}. We will also provide an application example, as well as another example where this method proves insufficient to realize a given couple (sign pattern, order of moduli).

\subsection{Concatenation of couples}\label{subsecconcat}
Consider a hyperbolic degree $d$ polynomial $V$ with distinct moduli of roots
and non-vanishing coefficients. Denote by $\Omega$ the order of the moduli of
its roots, where $\Omega$ is a string of letters $P$ and/or $N$. Then for
$\varepsilon >0$ small enough, the first $d+1$ coefficients of the
degree $d+1$ hyperbolic polynomials
$$W_-:=V(x)(x-\varepsilon )\;\mbox{and }\;W_+:=V(x)(x+\varepsilon )$$ are perturbations of
the respective coefficients of $V$. Hence they are of the same signs. The three polynomials realize the couples

$$V:=(\sigma (V),\Omega )~,~~~\, W_-:(\sigma (W_-),P\Omega )~~~\,
{\rm and}~~~\, W_+:(\sigma (W_+),N\Omega )~,$$
where $P\Omega$ and $N\Omega$ are the respective concatenations of strings.
Denote by $\alpha$ the last component of the sign pattern $\sigma (V)$, where
$\alpha =+$ or~$-$. Hence $\sigma (W_-)$ (resp. $\sigma (W_+)$)
is obtained from $\sigma (V)$ by adding to the right the component $-\alpha$
(resp. $\alpha$). We say that the couples $W_-$ and $W_+$ are obtained by
{\em concatenation} of the couple $V$ with the couples $((+,-),P)$ and
$((+,+),N)$ respectively.
\begin{ex}
{\rm

1)  The couple $(\Sigma_{3,1,2,1,1},~PPPPNNN)$ is realizable. Indeed, it is shown in \cite[1 of Theorem 1]{GKTMC} that for $d=6$, the couple  $(\Sigma_{3,1,2,1},PPPNNN)$ is realizable. Let then  $Q_6$ be a degree $6$ polynomial realizing the latter couple. Hence, for $\varepsilon > 0$ small enough, the product $Q_6(x)(x - \varepsilon)$ realizes the order $PPPPNNN$ with the sign pattern $\Sigma_{3,1,2,1,1}$.
\vspace{3mm}

2)  The couple $(\Sigma_{3,1,2,2},~ NPPPNNN)$ is realizable. Indeed, as mentioned above,  the couple  $(\Sigma_{3,1,2,1},PPPNNN)$  is realizable. Denote by $Q_6$ a degree $6$ polynomial realizing the latter couple. Hence, for $\varepsilon > 0$ small enough, the product $Q_6(x)(x + \varepsilon)$ realizes the order $NPPPNNN$ with the sign pattern $\Sigma_{3,1,2,2}$.
}
\end{ex}

\subsection{Algorithm and Numerical examples}
We apply the same idea as the one used in Algorithm~1 and modify it to match the problem as follows.
\begin{algorithm}[h]
\caption{Realizable couples (sign pattern, order of moduli) for a given degree~$d$}
\label{algo_ordermoduli}
\begin{algorithmic}[1]
    \STATE Input a degree~$d$, a sign pattern and an  order of moduli.
    \STATE Randomly generate $d$  independent, identically and  uniformly distributed positive numbers. 
		\STATE Sort the numbers then multiply by $-1$ when needed to construct the given order of moduli. 
		\STATE Calculate the coefficients of the hyperbolic polynomial using the ordered real roots from the previous step.
		\STATE Verify if the coefficients of the obtained polynomial correspond to the given sign pattern.
		\STATE Repeat until obtaining the good polynomial or until reaching an arbitrary parameter $N$.
\end{algorithmic}
\end{algorithm}

A first numerical example is the couple $(\Sigma_{3,4,1}, [0,0,5])$ whose  realizability cannot be obtained by any known analytic methods. A numerical example given by our algorithm implies  that the realizability of this couple is given by the polynomial 


$$\begin{array}{l}
(x-0.77)(x-4.28)(x+4.31)(x+4.47)(x+4.59)(x+4.68)(x+4.91)\\

=x^7+17.91x^6+98.1106x^5-21.793074x^4-1971.427200x^3-5976.303538x^2\\
-2955.965399x+6696.676474.\\
\end{array}$$
This allows  to conclude that the sign pattern $\Sigma_{3,4,1}$ is realizable only by the orders of moduli $[0,5,0]$, $[0,4,1]$, $[0,3,2]$, $[0,2,3]$, $[0,1,4]$, and $[0,0,5]$. Indeed, according to  \cite[Theorem 3]{KoPuMaDe}, an order of  moduli  is realizable only if $\alpha_1 < \gamma_1$, where $\alpha_1 < \alpha_2$ and $-\gamma_i$, $i=1,\dots,5$,  $\gamma_i < \gamma_{i+1}$ are respectively the positive and negative roots of a polynomial of degree $7$ with sign pattern $\Sigma_{3,4,1}$. Then, by applying part (2) of \cite[ Theorem 4]{KoPuMaDe}, we can prove that the moduli orders $[0,5,0]$, $[0,4,1]$, $[0,3,2]$, $[0,2,3]$, and $[0,1,4]$ are realizable.\\

Other numerical examples  concern the sign pattern  $\Sigma_{1,2,3,2}$ with  the following orders of moduli :
$$[0,3,0,1], [1,2,0,1], [2, 0, 1, 1], [2, 0, 0, 2], [2, 1, 0, 1]\;\mbox{ and }\; [3, 0, 0, 1] $$ whose realizability cannot be proved by any analytic method. We display here the numerical results:

$$\begin{array}{ll}
[0,3,0,1]:&(x-0.628)(x+0.688)(x+0.722)(x+0.950)(x-2.83)(x-4.26)(x+4.33)\\
&=x^7-1.028x^6-23.070064x^5+18.25165163x^4+85.39426303x^3\\&+32.00673579x^2
-30.03754531x-15.47008913\\\\

[1,2,0,1]:&(x+0.14)(x-0.15)(x+0.20)(x+0.32)(x-0.77)(x-2.05)(x+2.13)\\
&=x^7-0.18x^6-4.7422x^5+1.066232x^4+1.55397477x^3+0.1792075450x^2\\
&-0.03291572340x-0.004518803520\\\\

[2,0,1,1]:&(x+0.88)(x+1.19)(x-2.64)(x-2.67)(x+2.8)(x-3.69)(x+3.92)\\
&=x^7-0.21x^6-26.5337x^5+4.534365x^4+205.9891230x^3\\
&+14.83884381x^2
-467.7618374x-298.9615155\\\\
\end{array}$$
$$\begin{array}{ll}
[2,0,0,2]:&(x+1.01)(x+1.65)(x-3.3)(x-3.9)(x-4.23)(x+4.24)(x+4.47)\\
&x^7-0.06x^6-42.8452x^5+2.610486x^4+567.6094115x^3+68.3101894x^2\\
&-2166.332517x-1719.481913\\\\

[2,1,0,1]:& (x+1.13)(x+1.7)(x-3.28)(x+3.46)(x-3.559)(x-4.445)(x+4.64)\\
&=x^7-0.354x^6-40.362845x^5+7.46423375x^4+496.1523459x^3\\
&+96.0221457x^2-1867.359344x-1600.276550\\

[3,0,0,1]:&(x+1.19)(x+1.3)(x+1.4)(x-2)(x-3)(x-3.5)(x+3.93)\\
&=x^7-0.68x^6-22.6493x^5+11.98954x^4+135.291379x^3\\
&+16.6357660x^2-260.8328310x-178.7434740.\\\\
\end{array}$$

The arbitrary parameters $N$ and  $\ell$ are chosen equal to $10^3$ and $5$, respectively. The CPU time is about $1.12$ seconds.\\

These numerical results, combined with analytic methods allow  to deduce the following proposition:

\begin{prop}
The sign patterns $\Sigma_{1,2,3,2}$ is realizable  by exactly  $21$  out of $35$   possible  orders of moduli. 
All other orders of moduli are not realizable.
\end{prop}
\begin{proof}
Recall that a degree $7$ hyperbolic polynomial, with three sign changes in the sequence of its coefficients, has $35$  a priori possible orders of moduli $[u_1, u_2, u_3, u_4]$. Among them, there are $15$ cases where $u_4=0$, $10$ cases where $u_4=1$, $6$ cases where $u_4 = 2$, $3$ cases where $u_4=3$ and one case where $u_4$ equals $4$.

Consider a degree $7$ hyperbolic polynomial $Q$. Let $-\gamma_4 < -\gamma_3 < -\gamma_2 < -\gamma_1<0$ be its negative, and let $0 < \alpha_1 < \alpha_2 < \alpha_3$ be its positive roots. We have

 $$Q:= x^7 + \sum_{j=0}^{6} q_j x^j = (x - \alpha_1)(x - \alpha_2)(x - \alpha_3) \prod_{i=1}^{4} (x + \gamma_i)$$

 We give the proof of non-realizability of  $14$ orders  of moduli. 
The couple $(\Sigma _{1,2,3,2},[1,1,1,1])$ is not realizable, because the order of moduli $[1,1,1,1]$ is rigid and hence realizable only with the sign pattern $\Sigma _{2,2,2,2}$, see \cite[Definition 2 and Theorem 1]{KorigMO}.
    
We prove that the remaining $13$ cases are not realizable. We suppose that they are realizable by a polynomial $Q$, one obtains that

  $$q_6:=(\gamma _2-\alpha _1)+(\gamma _3-\alpha _2)+(\gamma _4-\alpha _3)+ \gamma_1>0,$$
  which contradicts the sign
  pattern. These are
\begin{equation}\label{13ineq}
\begin{array}{ll}
[0,0,0,4]:&\alpha _1<\alpha_2<\alpha _3<\gamma _1<\gamma _2<\gamma _3<\gamma_{4},\\

[0,0,1,3]:&\alpha_1<\alpha_2<\gamma _1<\alpha _3<\gamma _2<\gamma _3<\gamma_{4},\\ 

[0,0,2,2]:&\alpha_1<\alpha_2<\gamma_1<\gamma_2<\alpha_3<\gamma _3<\gamma_{4},\\

[0,0,3,1]:&\alpha _1<\alpha_2<\gamma _1<\gamma _2<\gamma _3<\alpha _3<\gamma_{4},\\ 

[0,1,0,3]:&\alpha _1<\gamma _1<\alpha_2<\alpha _3<\gamma _2<\gamma _3<\gamma_{4},\\

[0,1,1,2]:&\alpha _1<\gamma _1<\alpha_2<\gamma _2<\alpha _3<\gamma _3<\gamma_{4},\\ 

[0,1,2,1]:&\alpha _1<\gamma_1<\alpha_2<\gamma_2<\gamma_3<\alpha _3<\gamma _4,\\

[0,2,0,2]:&\alpha _1<\gamma_1<\gamma_2<\alpha_2<\alpha _3<\gamma _3<\gamma_{4},\\

[0,2,1,1]:&\alpha _1<\gamma_1<\gamma_2<\alpha_2<\gamma_3<\alpha _3<\gamma_{4},\\

[1,0,0,3]:&\gamma_1<\alpha _1<\alpha_2<\alpha _3<\gamma _2<\gamma _3<\gamma_{4},\\ 

[1,0,2,1]:&\gamma_1<\alpha _1<\alpha_2<\gamma_2<\gamma_3<\alpha _3<\gamma _4,\\

[1,0,1,2]:&\gamma_1<\alpha _1<\alpha_2<\gamma _2<\alpha _3<\gamma _3<\gamma_{4},\\ 

[1,1,0,2]:&\gamma_1<\alpha _1<\gamma_2<\alpha_2<\alpha _3<\gamma _3<\gamma_{4}.\\ 
 \end{array}
\end{equation}
It is shown in \cite[Subsection~3.5]{KoSoz19} that for $d=6$, the sign pattern $\Sigma_{2,3,2}$ is realizable with all 15 compatible orders $\Omega$ ($\Omega$ is any string of 2 letters $P$ and 4 letters $N$). Denote by $Q_6$  a polynomial realizing the couple $(\Sigma_{2,3,2},\Omega )$. Hence for $\delta >0$
  sufficiently large,   the product $T(x)(x-\delta )$ realizes the order $\Omega P$ with the sign pattern $\Sigma_{1,2,3,2}$, (see \ref{subsecconcat}). 
   This prove the realizability by concatenation of the sign pattern $\Sigma_{1,2,3,2}$ with the following $15$  orders of moduli:
  $$\begin{array}{lllll}
  [0, 0, 0, 4],&[0, 0, 1, 3],&[0, 0, 2, 2],&[0, 0, 3, 1],&[0, 0, 4, 0],\\
  
[0, 1, 0, 3],&[0, 1, 1, 2],&[0, 1, 2, 1],&[0, 1, 3, 0],&[0, 2, 0, 2],\\

[0, 2, 1, 1],&[0, 2, 2, 0],&[0, 3, 0, 1],&[0, 3, 1, 0],&[0, 4, 0, 0].
  \end{array}$$
The remaining six orders of moduli are realizable thanks to the algorithm as explained above. This ends the proof.
\end{proof}

\section{Problem 3: Distances between critical points and midpoints of zeros
of hyperbolic polynomials}

Let $P(x) = a_n x^n + \dots + a_1 x + a_0 $, with $ a_n \neq 0 $, an algebraic polynomial of degree $n $ with real coefficients $a_j $, $j=1\dots n$ and whose zeros $ x_1, \dots, x_n $ are all real.  Assume that $ x_1 \leq x_2 \leq \dots \leq x_n $ and let $ z_k = (x_k + x_{k+1})/2 $, which represents the midpoints of the zeros of $ P(x)$ . Define $\tilde{P}(x) := (x - z_1) \cdots (x - z_{n-1}) $. Let $ \xi_1 \leq \xi_2 \leq \dots \leq \xi_{n-1} $ be the zeros of $ P'(x) $, that is, the critical points of $P(x) $. We denote by $ m(P)$, $m(\tilde{P})$, and $m(P')$  the smallest distances between consecutive terms of the sequences ${x_k}$,  ${z_k}$ , and $ {\xi_k} $, respectively, that is:

$$ m(P) = \min \{ x_{k+1} - x_k: k = 1, \dots, n-1 \},$$
$$ m(\tilde{P}) = \min \{ z_{k+1} - z_k: k = 1, \dots, n - 2 \},$$
$$m(P') = \min \{ \xi_{k+1} - \xi_k: k = 1, \dots, n - 2 \}. $$

Similarly, we will denote by $ M(P)$, $M(\tilde{P}) $, and $ M(P')$ the corresponding maximum distances between the zeros of $P $, $\tilde{P}$ , and $P'$. The same notation will be used for entire functions that have only real zeros, with the convention that, instead of minima and maxima, we will consider infima and suprema whenever they are well-defined.
 We   answer here an open question in Theorem $3$ by D.K. Dimitrov and V.P. Kostov, published in \cite{DKDVPK}. This work analyzes the relationship between  $m(P)$  and  $m(\tilde{P})$ , as well as the one between $M(P) $ and $ M(\tilde{P})$. The analytic methods employed by Kostov and Dimitrov is build on a classical result by Marcel Riesz. 

We recall that the set of real entire functions of order at most two that have only real zeros is referred to as the Laguerre-P\'olya class. Each function belonging to this class will be   referred to as an $\mathcal{LP}$-function. Given that the issue raised by Farmer and Rhoades specifically concerns functions of order one, we will limit our discussion to the subclass of $\mathcal{LP}$-functions of order one.

The Laguerre-P\'olya class thus includes entire functions of order at most two. Motivated by the behavior of the zeros of the Riemann zeta function and its derivatives, Farmer and Rhoades \cite{DWF.RR} extended Riesz' result to entire functions belonging to the subclass $\mathcal{LP}1\subset \mathcal{LP} $. The functions of the class $\mathcal{LP}1$ are the ones which are uniform limits on compact sets of real polynomials with all roots real and of the same sign. They are entire functions of order 0 or 1.   It was proven in \cite{DWF.RR} (see Theorem 2.3.1) that for any function $f \in \mathcal{LP}1$  with zeros $x_k$ arranged in increasing order, if  $\xi < \eta$  are consecutive zeros of  $f + af$ , with $a \in \mathbb{R}$, then the following inequalities hold:

$$ \inf\{x_{k+1} - x_k\} \leq \eta - \xi \leq \sup\{x_{k+1} - x_k\}. $$

After obtaining several significant results regarding the distribution of the zeros of the entire functions associated with $\mathcal{LP}1$, and applying these to the Riemann $\xi$ function, the authors of \cite{DWF.RR} state Conjecture 5.1.1 concerning the inequality  between $m(P)$ and $m(\tilde{P})$. This conjecture can be formulated as follows:

\begin{conj}[Conjecture 5.1.1 of \cite{DWF.RR}]\label{conjecture} 
 Suppose that $P \in \mathcal{LP}1$ and that its zeros $x_k$ are listed in increasing order. If $\xi <\eta$
are consecutive zeros of $P'$,  then
\begin{equation}\label{conj}
{\rm inf} \{ (x_{k+2}-x_k)/2\}<\eta-\xi<{ \rm sup}\{(x_{k+2} -x_k)/2\}.
\end{equation}
\end{conj}
\vspace{3mm}

Observe that 
$$\begin{tabular}{lll}
$(x_{k+2 }-x_k)/2=z_{k+1}- z_k$, && ${\rm inf} \{ (x_{k+2}-x_k)/2\}=m(\tilde{P})$\\ {\rm }and && ${\rm sup} \{ (x_{k+2}-x_k)/2\}=M(\tilde{P})$.\\ 
\end{tabular}$$

\vspace{0.5cm}
We denote by $L_{-}R_{+}$ the case when the left
inequality in (\ref{conj}) fails while the right one holds; in a similar way we define the cases $L _{-} R_{-}$,
$L_{+}R_{-}$ and $L_{+}R_{+}$.
\vspace{3mm}

The answer to Conjecture \ref{conjecture} was provided by D.K. Dimitrov and V.P. Kostov in Theorems 1 and 2 of \cite{DKDVPK}. However, there remain two inequalities in point 5 of Theorem 2 for which the solution is still unknown. In other words, they have not managed to provide an analytic answer to the conjecture \ref{conjecture} regarding degrees 5 and 6 concerning the inequalities $L_{-}R_{+}$. The resolution of the case $L_{-}R_{+}$ for degree 6 using the same numerical approach leads to the following numerical example:

$$ \begin{array}{lll}

P&=&x^6-1.60x^5+0.5300x^4+0.122578x^3-0.03793509x^2-0.0025040322x\\
&&+0.000600530112
\end{array}$$
Its roots are equal to $x_1=-0.19$, $x_2=-0.18$, $x_3=0.13$, $x_4=0.21$, $x_5=0.67$, $x_6=0.96$. That gives
 $$\begin{array}{ll}
 (x_3-x_1)/2=0.16, &(x_4-x_2)/2=0.195,\\
 (x_5-x_3)/2=0.27& (x_6-x_4)/2=0.375,\\
 \end{array}$$
 and
$$\begin{array}{ll}
 m(\tilde{P})={\rm inf} \{ (x_{k+2}-x_k)/2\}=0.16,&k=1,\dots, 4\\
  M(\tilde{P})={\rm sup}\{ (x_{k+2}-x_k)/2\}=0.375,&k=1,\dots, 4.
 \end{array}$$
 The  derivative  of $P$ is
 $$\begin{array}{lll}
 P'&:=&6x^5-8x^4+2.12x^3+0.367734x^2-0.07587018x-0.0025040322,\\
  \end{array}$$
its  roots are $\xi_1=-0.1850968062$, $\xi_2=-0.02957083052$, $\xi_3=0.1718593928$, $\xi_4= 0.5155057599$ and $\xi_5=0.8606358173$. The difference between the consecutive roots of $P'$ are
$$\begin{array}{ll}
\xi_2-\xi_1=0.1555259757,&\xi_3-\xi_2=0.2014302233,\\
\xi_4-\xi_3=0.3436463671,&\xi_5-\xi_4=0.3451300574.
\end{array}$$
We have $m(P')=(\xi_2-\xi_1)=0.1555259757,~{\rm and}~M(P')=(\xi_5-\xi_4)=0.3451300574$ hence
$$ \begin{array}{lr}
 m(P') < m(\tilde{P})<M(P') <M(\tilde{P}).
\end{array}   $$
Therefore, the polynomial $P$  realizes the case $L_{-}R_{+}$.
The CPU time here is equal to 23.99 seconds for $N=10^5$ and $l=1$.\\
This confirms the robustness of our numerical method and proves  that this approach can be used to solve various mathematical problems.

Regarding the case $ L_{-}R_{+} $  mentioned in point (5) of Theorem 2 in \cite{DKDVPK}, the code produced no results even for draws of numbers larger than $ 10^{8} $. This reinforces the hypothesis that there is no polynomial of degree  $5 $  satisfying the two inequalities of the conjecture \ref{conjecture}.
\section{Conclusion}
In conclusion, the proposed numerical approach based on the generation of independent and uniformly  distributed random roots  to construct random polynomial offer an effective and efficient method to solve accurate polynomial analysis problem. When no result is given by the algorithm,  one can identify a priori non-realizable cases with a high degree of certainty.  This method is a powerful tool for theoretical investigations. Its simplicity, ease of implementation and computational speed make it an invaluable resource for researchers working in this field.

{\bf Acknowledgement} 
Authors would like to  thank Pr. Maher BERZIG for his helpful comments.

\end{document}